\documentclass[12pt,a4paper]{amsart}
\usepackage{hyperref}
\usepackage[numbers,sort & compress]{natbib}
\usepackage{graphics,layout,indentfirst,amsmath,amsthm,amssymb}
\newtheorem{theorem}{Theorem}[section]
\newtheorem{remark}[theorem]{Remark}
\newtheorem{corollary}[theorem]{Corollary}
\theoremstyle{theorem}

\newtheorem{lemma}[theorem]{Lemma}
\theoremstyle{definition} \theoremstyle{definition}
\newtheorem{defn}[theorem]{Definition}

\numberwithin{equation}{section} \makeatletter
\@namedef{subjclassname@2010}{2010 Mathematics Subject
Classification} \makeatother
\begin{document}
\title[Coefficient estimates of certain subclasses of analytic functions associated 
...]{Coefficient estimates of certain subclasses of analytic functions associated with Hohlov
operator}

 \author[P. Gochhayat]{P. Gochhayat}
\author[A. Prajapati]{A. Prajapati}
  \address{Department of Mathematics \\
   Sambalpur University\\
   Jyoti Vihar 768019\\
   Burla, Sambalpur, Odisha\\ India}
   \email{pgochhayat@gmail.com}
   \email{anujaprajapati49@gmail.com}
    \author[A. K. Sahoo ]{A. K. Sahoo}
    \address{Department of Mathematics \\
    	Veer Surendra Sai University of Technology \\
    	Sidhi Vihar 768018\\
    	Burla, Sambalpur, Odisha\\ India}
    \email{ashokuumt@gmail.com}
   \date{}
\thanks{The present investigation of the second author is supported
under the   INSPIRE fellowship, Department of Science and
Technology, New Delhi, Government of India, Sanction Letter No.
REL1/2016/2/2015-16.} 

\begin{abstract} A typical quandary in geometric functions theory is to study a functional composed
of amalgamations of the coefficients of the pristine function.
Conventionally, there is a parameter over which the extremal value
of the functional is needed.  The present paper deals with
consequential functional of this type. By making use of linear
operator due to Hohlov \cite{6}, a new subclass
$\mathcal{R}_{a,b}^{c}$ of analytic functions defined in the open
unit disk is introduced. For both real and complex parameter, the
sharp bounds for the Fekete-Szeg\"{o} problems are found. An attempt
has also been taken to found the sharp upper bound to the second and third 
Hankel determinant  for functions belonging to this class. All the
extremal functions are express in term of Gauss hypergeometric
function and convolution. Finally, the sufficient condition for
functions to be in $\mathcal{R}_{a,b}^{c}$ is derived.
 Relevant connections of the new results with well known ones are pointed out.
\end{abstract}


\maketitle{}
\section{Introduction and Preliminaries}\label{sec1}
Let $\mathcal{A}$  be  the class  of  functions  analytic  in the
\textit{open} unit disk
\begin{gather*}
\mathcal {U}:=\{z:z\in\mathbb{C}~\text{and}~|z|< 1\},
\end{gather*}
normalized by the condition $\mathit{f}(0)=0, ~~\mathit{f}^{\prime}(0)=1$ and has the Taylor-Maclaurin  series of the form:
\begin{equation}\label{p1}
\mathit{f}(z)=z+\sum_{k=2}^{\infty}a_{k}z^{k},~~~~~~~~~\qquad(z \in \mathcal{U}).
\end{equation}
Let $\mathcal S$ be the subclass of $\mathcal A$ consisting of
univalent functions. Suppose that $f$ and $g$ are in $\mathcal A$.
We say that $f$ is \textit{subordinate} to $g$, (or $g$ is
\textit{superordinate} to $f$), write as $$f \prec g ~\text{
in}~\mathcal U~\text{ or}~ f(z) \prec g(z) \qquad\qquad (z \in
\mathcal U),$$ if there exists a function $\omega\in \mathcal A$,
satisfying the conditions of the Schwarz lemma
$(~\text{i.e.}~\omega(0)=0$ and $|\omega(z)|< 1)$ such that
$$f(z)=g(\omega(z))\qquad\qquad(z \in \mathcal U).$$ It follows that
\begin{gather*}
f(z) \prec g(z)\;(z \in \mathcal U) \Longrightarrow f(0)=g(0) \quad
\text{and} \quad f(\mathcal U) \subset g(\mathcal U).
\end{gather*}
In particular, if $g$ is \textit{univalent} in $\mathcal U$, then
the reverse implication also holds (cf.\cite{milmo1}). Denote
$\mathcal{P}$, the class of functions $\phi$ which is analytic  in
$\mathcal {U}$ and is of the form
\begin{equation}\label{p2}
    \phi(z)=1+p_{1}z+p_{2}z^{2}+\cdots\qquad (z \in \mathcal{U}),
\end{equation}
with $\phi(0)=1$ and $\Re(\phi(z))>0$.

If $f$ and $g$ are functions in $\mathcal{A}$ and given by the
power series
\begin{equation*}
f(z) = z+\sum_{n=2}^\infty a_nz^n ~~\text{and}~~ g(z) =z+
\sum_{n=2}^\infty b_nz^n ~~~~~~~~\qquad(z\in \mathcal{U}),
\end{equation*}
then the Hadamard product (or \textit{Convolution}) of $f$ and $g$
denoted by $f * g$, is defined by
\begin{equation*}
(f * g)(z) = z+\sum_{n=2}^\infty a_nb_nz^n = (g * f)(z) ~~~~~~~~~\qquad(z\in
\mathcal{U}).
\end{equation*}
Note that $f*g\in \mathcal{A}$.

For the complex parameters $a$, $b$ and $c$ with $ c \neq
0,-1,-2,-3,\cdots,$ the Gauss hypergeometric function denoted by $ _{2}F_{1}(a,b,c;z)$ and is defined by
\begin{eqnarray*}_{2}F_{1}(a,b;c;z)=\sum_{n=0}^{\infty}\frac{(a)_{n}(b)_{n}}{(c)_{n}}\frac{z^{n}}{n!}\qquad(z \in \mathcal
U),
    \end{eqnarray*}
    where $(\alpha)_{n}$ denotes the Pochhammer symbol (or shifted factorial) given in terms of the Gamma function $\Gamma,$ by
\begin{eqnarray*}\label{p4}
    (\alpha)_{n}={\Gamma(\alpha+n)\over \alpha}=\begin{cases}1;
        &\text{if}~~n=0,\\\alpha(\alpha+1)(\alpha+2)\cdots(\alpha+n-1);&\text{if}~~n\neq0.\end{cases}\end{eqnarray*}

In terms of Gauss hypergeometric function and convolution, Hohlov
(cf.\cite{6}, \cite{7}) introduced and studied a linear operator
denoted by $\mathcal{I}_{a,b}^{c}$ and defined by
$\mathcal{I}^{c}_{a,b}f:\mathcal{A}\rightarrow \mathcal{A},$ as
\begin{eqnarray}\label{p3}
    \mathcal{I}^{c}_{a,b}f(z)&:=&z{} _{2}F_{1}(a,b;c;z)*f(z) \qquad(z \in
    \mathcal{U}).
\end{eqnarray}

We note that upon suitable choice of the parameters, the above
defined three-parameter family of  operator unifies various other
linear operators which are introduced and studied earlier. For
example
\begin{enumerate}
    \item $\mathcal{I}_{a,1}^{c}:=\mathcal{L}(a,c)$, the well known Carlson-Shaffer  operator (cf. \cite{2}).
    \item $\mathcal{I}_{\lambda+1,1}^{1}:=\mathcal{D}^{\lambda}~(\lambda >-1)$, is the  Ruscheweyh derivative operator of order $\lambda$ (cf. \cite{23}).
    \item $\mathcal{I}_{1,1+\eta}^{2+\eta}:=\mathcal{I}_{B}^{\eta}$, the well known Bernardi  integral  operator (cf. \cite{1}, also see \cite{owasri}).
    \item $\mathcal{I}_{2,1}^{2-\alpha}:=\Omega^{\alpha}_z$, the fractional differential operator (cf. \cite{22}), also renamed as Owa-Srivastava fractional differential operator (cf. \cite{16, 17,18}).
    \item $\mathcal{I}_{2,1}^{n+1}:=\mathcal{I}_{n}$, the Noor integral operator (cf. \cite{21}, also see \cite{19}).
    \item $\mathcal{I}_{\mu, 1}^{\lambda+1} := \mathcal{I}_{\lambda, \mu},$ the well known Choi-Saigo-Srivastava operator (cf. \cite{3}).
    \item $\mathcal{I}_{1, 2}^{3} := \mathcal{L},$ is the Libera
    integral operator (see \cite{owasri}).
\item $\mathcal{I}_{2, 1}^{1} := \mathcal{I}_A,$ is the Alexander transformation, where as $\mathcal{I}_{1, 1}^{2}=\int_{0}^z {f(t)\over t} dt$ is its inverse transform (see \cite{4}).
\end{enumerate}
From $(\ref{p3})$ it is clear  that
\begin{equation}\label{p5}
    z(\mathcal{I}_{a,b}^{c}f(z))^{\prime}=a(\mathcal{I}_{a+1,b}^{c}f(z))-(a-1)\mathcal{I}_{a,b}^{c}f(z).
\end{equation}
By using Hohlov operator, we now defined a new subclass of $\mathcal A$ as follows:
\begin{defn}
    A function $f \in \mathcal{A}$ is said to be in the class $\mathcal{R}_{a,b}^{c},$ if and only if  $\left(\frac{\mathcal{I}_{a,b}^{c}f}{z}\right)(z)$ takes all values to the bounded region by the right half plane of the lemniscate of Bernoulli given by,
    \begin{equation*}\label{p8}
    \{w \in \mathbb{C}:\mid  w^{2}-1 \mid<1\}=\{u+iv:(u^{2}+v^{2})^{2}=2(u^{2}-v^{2})\}.
    \end{equation*}
In terms  of subordination, we have  $f \in \mathcal{R}_{a,b}^{c}$ if it satisfies
    \begin{equation*}\label{p7}
    \frac{(\mathcal{I}_{a,b}^{c}f)(z)}{z}\prec \sqrt{(1+z)},\qquad (z \in \mathcal{U}).
    \end{equation*}
\end{defn}
\begin{remark}\label{pg1}
    Taking $b=1$, the class $\mathcal{R}_{a,1}^{c} :=\mathcal{R}(a,c)$, recently introduced and studied by Patel and Sahoo
    \cite{22a}.
    \end{remark}
\begin{remark}\label{pg2}
    Taking $a=2,~b=1$ and $c=1$, we say a function $f$ given by
    (\ref{p1}) is in the class $\mathcal{R}_{2,1}^{1}$ if it
    satisfies the subordination relation
    $$f'(z)\prec \sqrt{(1+z)}, \qquad (z\in\mathcal U).$$
    The family $\mathcal{R}_{2,1}^{1}$ is recently studied by Sahoo and Patel \cite{s1} which is close-to-convex and hence univalent.
    \end{remark}
    \begin{remark}\label{pg3}
    Taking $a=1,~b=1$ and $c=1$, we say a function $f$ given by
    (\ref{p1}) is in the class $\mathcal{R}_{1,1}^{1}$ if it
    satisfies the subordination relation
    $${f(z)\over z}\prec \sqrt{(1+z)}, \qquad (z\in\mathcal U).$$
    \end{remark}
Note that, the family 
$\mathcal{R}_{1,1}^{1} $ 
 contain univalent as well as non
 univalent functions (cf. \cite{4}).

  It is well known
that the $n^{th}$ coefficient of function belonging to the class
$\mathcal S$ is bounded by $n$ and the bounds for the coefficients
gives information about the geometric properties of the functions.
For example, the $n^{th}$ coefficient gives information about the
area  where as the second coefficient of functions in the family
$ S$ yields the growth  and distortion properties of
function. A typical problem in geometric function theory is to study
a functional made up of combinations of the coefficients of the
original function. Usually, there is a parameter over which the
extremal value of the functional is needed. Some of our results
deals with one important functional of this type: the
Fekete-Szeg\"{o} functional. The classical problem settled by
Fekete-Szeg\"{o} \cite{5} is to find for each $\lambda\in[0,1]$ the
maximum value of the coefficient functional is defined by
$\Phi_\lambda (f):=|a_3-\lambda a_2^2|$ over the class $\mathcal S$
and was proved by using Loewner method. Several researchers solved the Fekete-Szeg\"{o}
problem for various subclasses of the class of $\mathcal S$ and
related subclasses of functions in $\mathcal A$. For instant see
\cite{1a}, \cite{dziok}, \cite{kanas,9,10,11,11a}, 
\cite{17}, \cite{19a}, etc. For a systematic survey on Fekete-Szeg\"{o}
problem of   classical subclasses
of $\mathcal S$ we refer \cite{24}.  In \cite{24}, Srivastava et al.
held that the inequality was sharp, however recently Peng (cf. \cite{24a}) showed that the extremal function given there for the case of $\mu \in (2/3,1]$ is not sharp.  
Cho et al. \cite{cho} obtained
Fekete-Szeg\"{o} inequalities for close-to-convex function with respect to a certain convex function which improve the bound studied in \cite{24}.

Another way to investigate the sharp bound for the non linear
functional is by using Hankel or Toeplitz determinant. Recalling
  the $q^{th}$ Hankel determinant of $f$ for
$q\geq 1$ and $n\geq 1$  which is introduced and studied by Noonan
and Thomas \cite{20} as
\begin{equation}\label{p9}
    H_{q}(n)=
    \begin{vmatrix}
        a_{n}&a_{n+1}& \cdots & a_{n+q-1}\\a_{n+1}& a_{n+2}&\cdots & a_{n+q}\\ \vdots & \vdots & \vdots & \vdots\\a_{n+q-1} & a_{n+q} & \cdots & a_{n+2q-2}
    \end{vmatrix}
    \qquad (q,n \in \mathbb{N} ).
\end{equation}
This determinant has been studied by several authors including Noor
\cite{88} with the subject of inquiry ranging from the rate of
growth of $H_{q}(n)$ ~~(as $n \rightarrow \infty$) to the
determinant of precise bounds with specific values of $n$ and $q$
for certain subclasses of analytic functions in the unit disk
$\mathcal{U}.$ For $q=2,~~n=1,~~a_{1}=1,$ then the Hankel
determinant simplifies to $H_{2}(1)=\mid  a_{3}-a_{2}^{2}\mid.$ For
$n=q=2,$ then the Hankel determinant  simplifies to $H_{2}(2)=\mid
a_{2}a_{4}-a_{3}^{2} \mid$. The Hankel determinant $H_{2}(1)$ was
considered by Fekete and Szeg\"{o} \cite{4} and refer to $H_{2}(2)$
as the second Hankel determinant. It is known (cf. \cite{5}) that if $f$ is univalent  in $\mathcal{U}$ then the sharp upper inequality
$H_{2}(1)=\mid a_{3}-a_{2}^{2}\mid \leq 1$ holds. In \cite{77}, Janteng et al. obtained sharp bounds for the
functional $H_{2}(2)$ for the function $f$ in the subclass $RT$ of $\mathcal S$, consisting of functions whose derivative has a
positive real part introduced by Mac Gregor \cite{mac}. They shown
that for every $f\in RT$, $H_{2}(2)=\mid a_{2}a_{4}-a_{3}^{2}
\mid\leq 4/9$. They also found the sharp second Hankel determinant for
the classical subclass of $\mathcal S$, namely, the  class of starlike and
convex functions respectively denoted by $\mathcal S^*$ and
$\mathcal K$ (cf. \cite{777}). The bounds obtained for these two
classes are $\mid a_{2}a_{4}-a_{3}^{2} \mid\leq 1$ and $\mid
a_{2}a_{4}-a_{3}^{2} \mid\leq 1/8$ respectively.  Recently, Ready
and Krishna \cite{ram} obtained the Hankel determinants for starlike
and convex functions with respect to symmetric points. Lee et al.
\cite{lee} obtained the second Hankel determinant for functions
belonging to subclasses of Ma-Minda starlike and convex functions.
Using Owa-Srivastava \cite{22} operator, Mishra and Gochhayat
\cite{16} have obtained the sharp bound to the non-linear functional
$|a_{2}a_{4}-a_{3}^{2}|$ for the subclass of analytic functions
denoted by $R_{\lambda}(\alpha,p)~~(0\leq p \leq 1, 0\leq \lambda <
1, |\alpha|<\frac{\pi}{2}),$ defined as $\Re
\left(e^{i\alpha}\frac{\Omega_{z}^{\lambda}f(z)}{z}\right) > p \cos
\alpha$. Similar coefficient bounds are obtained for
various other subclasses of analytic functions which is defined by using suitable linear operators (see \cite{1aa}, \cite{12},
\cite{kund}, \cite{19b}, \cite{YAVUZ},  \cite{YAVUZ1}, etc).
 We also consider the Hankel determinant in the case of $q=3$ and $n=1,$ denoted by $H_{3}(1),$ given by
 $H_{3}(1)=a_{3}(a_{2}a_{4}-a_{3}^{2})-a_{4}(a_{4}-a_{2}a_{3}+a_{5}(a_{3}-a_{2}^{2}).$
Clearly,
\begin{eqnarray}\label{anuja}
|H_{3}(1)| \leq |a_{3}| |a_{2}a_{4}-a_{3}^{2}|+|a_{4}| |a_{2}a_{3}-a_{4}|+|a_{5}||a_{3}-a_{2}^{2}|.
\end{eqnarray}
 In \cite{2b}, Babalola showed that all the  functional on right hand side of (\ref{anuja}) is sharp for function belongs to the class $\mathcal{RT},$ $\mathcal{S}^{*}$ and $\mathcal{K}.$ Recently Bansal et. al. \cite{2bb} and Raza and Malik \cite{r1} obtained the bound $|H_{3}(1)|$ for certain subclasses of analytic univalent functions.

In our present  investigation, following the techniques adopted by
Libera and Zlotkiewicz (cf.\cite{13}, \cite{14}), for functions
belongs to the family $\mathcal R_{a,b}^c$, the Fekete-Szeg\"{o}
problem is completely solved for both real and complex parameter.
All the extremal functions are presented in terms of Gauss
Hypergeometric functions and convolution. Secondly, using the
techniques of Hankel determinant, the sharp upper bound for the non
linear functional $|a_2a_4-a_3^2|$ is derived.
Motivated by the work of Babalola \cite{2b}  we found the sharp upper bound  to the  $|H_{3}(1)|$ for the function belonging to the  class $\mathcal{R}_{a,b}^{c}$ related with lemniscate of Bernoulli.
 Sufficient condition
for functions to be in $\mathcal{R}_{a,b}^{c}$ is also presented.

To establish our main results, we need the following lemmas:
\begin{lemma}\label{k1}(cf. \cite{4}, \cite{13,14,15})
    Let the function $\phi \in \mathcal{P},$ given by $(\ref{p2})$.  Then
    \begin{equation}\label{p10}
        |p_{k}| \leq 2 \qquad(k \geq 1),
    \end{equation}
    \begin{equation}\label{p11}
        | p_{2}-\nu p_{1}^{2}| \leq 2~~ max \{1, \mid 2\nu-1 \mid\},\qquad (\nu \in \mathbb{C}),
    \end{equation}
     \begin{equation}\label{p13}
        p_{2}=\frac{1}{2}\big\{p_{1}^{2}+(4-p_{1}^{2})x\big\},
    \end{equation}
    \begin{eqnarray}\label{p14}
        p_{3}=\frac{1}{4} \left\{p_{1}^{3}+2(4-p_{1}^{2})p_{1}x-(4-p_{1}^{2})p_{1}x^{2}+2(4-p_{1}^{2})(1-\mid x \mid ^{2})z\right\},
    \end{eqnarray}
    for some complex numbers $x,$ $z$ satisfying $|x| \leq 1$ and $|z| \leq1.$
    The estimates in $(\ref{p10})$ and $(\ref{p11})$ are sharp for the functions given by
    \begin{equation*}\label{p12}
        f(z)=\frac{1+z}{1-z} ,\qquad g(z)=\frac{1+z^{2}}{1-z^{2}}\qquad (z\in \mathcal{U}).
    \end{equation*}
\end{lemma}

   
\begin{lemma}(cf.\cite{15})\label{b1}
    Let $\phi \in \mathcal{P}$  and of the form (\ref{p2}), then
    \begin{eqnarray*}\label{p15}
        \mid p_{2}-\nu p_{1}^{2}\mid\leq
        \begin{cases}-4\nu+2
            &; \text{if}~~\nu<0,\\2&; \text{if}~~0\leq\nu\leq1,\\\ 4\nu-2&; \text{if}~~\nu>1.
        \end{cases}
    \end{eqnarray*}
    For $\nu < 0$ or $\nu > 1,$ equality holds if and only if $\phi(z)$ is $\frac{1+z}{1-z}$ or one of its rotations.
    For $0 < \nu < 1,$ the equality holds if and only if $\phi(z)=\frac{1+z^{2}}{1-z^{2}}$ or one of its rotations.
    For $\nu=0,$ the equality holds if and only if $\phi(z)=(\frac{1}{2}+\frac{\eta}{2})\frac{1+z}{1-z}+(\frac{1}{2}-\frac{\eta}{2})\frac{1-z}{1+z}~~~~(0\leq \eta \leq1)$ or one of its rotations. For  $\nu=1,$ the equality holds if and only if $\phi$ is the reciprocal of one of the functions  such that the equality holds in the case $\nu=0$.

    Moreover, when $0<\nu<1,$ although the above upper bound is sharp, it can also be improved as follows:
    $$\mid p_{2}-\nu p_{1}^{2}\mid+\nu \mid p_{1} \mid^{2} \leq2,\qquad(0<\nu\leq \frac{1}{2}),$$
    and
    $$\mid p_{2}-\nu p_{1}^{2}\mid + (1-\nu)\mid p_{1} \mid^{2}\leq2, \qquad(\frac{1}{2}<\nu\leq1).$$

\end{lemma}

\section{Main Results}
Unless otherwise mentioned, throughout this sequel we assume that
both $a,~b \geq c >0$. We begin with the proof of Fekete-Szeg\"{o}
problem for the class $\mathcal{R}_{a,b}^{c}$.

\begin{theorem}\label{k1}
    If the function $f$ given by $(\ref{p1})$ belongs to the class $\mathcal{R}_{a,b}^{c},$ then for any $\mu \in \mathbb{C},$
    \begin{equation}\label{p16}
        \mid a_{3}- \mu  a_{2}^{2} \mid \leq \frac{(c)_{2}}{(a)_{2}(b)_{2}} max \left\{1, \frac{\mid ab(c+1)+\mu c (a+1)(b+1)\mid}{4a b(c+1)}\right\}.
    \end{equation}
    The estimate (\ref{p16}) is sharp.
    \begin{proof}
        If $f \in \mathcal{R}_{a,b}^{c}$ then by the definition of $\mathcal{R}_{a,b}^{c},$ satisfies the condition
        \begin{equation}\label{p17}
            \left|\left(\frac{\mathcal{I}_{a,b}^{c}f(z)}{z}\right)^{2}-1 \right| <1,~~~~\qquad(z \in
            \mathcal{U}).
        \end{equation}
        So now using $(\ref{p17})$ and definition of subordination that satisfies, the relation
        \begin{equation}\label{p18}
            \frac{\mathcal{I}_{a,b}^{c}f(z)}{z}= \sqrt{1+w(z)},
        \end{equation}
        where $w$ is analytic in $\mathcal{U}$ and satisfies the conditions of Schwarz lemma $w(0)=0$ and $\mid w(z) \mid <1$.
        Setting
        \begin{equation*}\label{p19}
            \phi(z)=\frac{1+w(z)}{1-w(z)}=1+p_{1}z+p_{2}z^{2}+\cdots,\qquad(z \in \mathcal{U}),
        \end{equation*}
        implies that $\phi \in \mathcal{P}.$ From the above expression, we get
        \begin{equation*}\label{p20}
            w(z)=\frac{\phi(z)-1}{\phi(z)+1}\qquad(z \in
            \mathcal{U}).
        \end{equation*}
        Therefore,  $(\ref{p18})$ gives
        \begin{equation}\label{p21}
            \frac{\mathcal{I}_{a,b}^{c}f(z)}{z}=\left(\frac{2\phi(z)}{1+\phi(z)}\right)^{1/2}\qquad(z \in \mathcal{U}).
        \end{equation}
        Which upon simplification and comparing the co-efficient of $z,$ $z^{2},$ $z^{3}$ both side of  $(\ref{p21})$
        yields.
        \begin{equation}\label{p24}
            a_{2}=\frac{c}{4ab}p_{1},
        \end{equation}
        \begin{equation}\label{p25}
            a_{3}=\frac{2(c)_{2}}{(a)_{2}(b)_{2}}\left(\frac{1}{4}p_{2}-\frac{5}{32}p_{1}^{2}\right),
        \end{equation}
        and
        \begin{equation}\label{p26}
            a_{4}=\frac{6(c)_{3}}{(a)_{3}(b)_{3}}\left(\frac{1}{4}p_{3}-\frac{5}{16}p_{1}p_{2}+\frac{13}{128}p_{1}^{3}\right).
        \end{equation}
        Thus, by using $(\ref{p24})$ and $(\ref{p25}),$ we get
        \begin{equation}\label{p27}
            \mid a_{3}- \mu  a_{2}^{2}\mid=\frac{1}{2}\frac{(c)_{2}}{(a)_{2}(b)_{2}}\left|p_{2}-\frac{5ab(c+1)+\mu c (a+1)(b+1)}{8ab(c+1)}p_{1}^{2}\right|.
        \end{equation}
        Now using $(\ref{p11})$ in $(\ref{p27}),$ we get
        \begin{equation*}\label{p28}
            \mid a_{3}- \mu  a_{2}^{2}\mid\leq\frac{(c)_{2}}{(a)_{2}(b)_{2}} \max \left\{1, ~\frac{\mid ab(c+1)+ \mu c (a+1)(b+1)\mid}{4ab(c+1)}\right\}.
        \end{equation*}
        The estimate $(\ref{p16})$ is sharp for the function $f \in \mathcal{A}$ defined in $\mathcal{U}$ by
        \begin{eqnarray*}
            f(z)=\begin{cases}z _{1} F_{2}(c,b,a;z)* 4z\sqrt{1+z^{2}};
                &\text{}~~\frac{ |ab(c+1)+\mu c (a+1)(b+1)|}{4ab(c+1)}\leq 1,\\ z _{1}F_{2}(c,b,a;z)* \frac{-z}{2}(1+z)^{-1};&\text{}~~ \frac{ |ab(c+1)+\mu c (a+1)(b+1)|}{4ab(c+1)}>1.\end{cases}\end{eqnarray*}
        This completes the proof of Theorem \ref{k1}.
    \end{proof}
\end{theorem}
For $\mu=1,$ the bound  $|H_{2}(1)|$ directly follows from Theorem \ref{k1}.
\begin{corollary}\label{a2}
If the function $f$ given by $(\ref{p1})$ belongs to the class $\mathcal{R}_{a,b}^{c},$ then
 \begin{equation*}
        |H_{2}(1)|=|a_{3}-a_{2}^{2}|\leq \frac{(c)_{2}}{(a)_{2} (b)_{2}} .
    \end{equation*}
\end{corollary}
\begin{remark}\label{k2}
    Putting $b=1$ in Theorem $\ref{k1}$, we get the sharp bound for the function belonging to the subclass of $\mathcal A$ associated with   the Carlson-Shaffer operator (cf. \cite{22a}, Theorem
    3). 
\end{remark}
Further, by specializing the parameters $a$, $b$ and $c$ we have the
following sharp bounds:
\begin{remark}\label{k3}
    Putting $a=\lambda +1,$ $ b=1,$ $c=1$ in the Theorem $\ref{k1}$, the required sharp bound for the function belonging to the subclass of $\mathcal A$ associated with Ruscheweyh
    operator is given by
    \begin{equation*}
        \mid a_{3}- \mu  a_{2}^{2}\mid \leq \frac{1}{(\lambda+1)_{2}} \max\left\{1, ~\frac{|(\lambda+1)+\mu (\lambda+2)|}{4(\lambda+1)}\right\}.
    \end{equation*}
\end{remark}
\begin{remark}\label{k4}
    Putting $a=1,$ $b=1+\eta,$ $c=2+\eta$ in the Theorem $\ref{k1}$, the required sharp bound for the function belonging to the subclass of $\mathcal A$ associated with Bernardi
    operator is given by
    \begin{equation*}
        \mid a_{3}- \mu  a_{2}^{2}\mid \leq \frac{(3+\eta)}{2(1+\eta)} \max \left\{1, ~\frac{|(1+\eta)(3+\eta)+2\mu (2+\eta)^{2}|}{4(1+\eta)(3+\eta)}\right\}.
    \end{equation*}
\end{remark}
\begin{remark}\label{k5}
    Putting $a=2,$ $b=1$ $c=1$ in the Theorem $\ref{k1}$, the required sharp bound for the function belonging to the subclass of $\mathcal A$ associated with Alexander differential
    operator (cf. \cite{s1}, Theorem 2.1) is given by
    \begin{equation*}
        \mid a_{3}- \mu  a_{2}^{2}\mid \leq \frac{1}{6} \max \left\{1, ~\frac{|2+3\mu|}{8}\right\}.
    \end{equation*}
\end{remark}
\begin{corollary}\label{k6}
    If the function $f$ given by $(\ref{p1})$ belongs to the class $\mathcal{R}_{a,b}^{c},$ then it follows from $(\ref{p24})$ that $\mid a_{2} \mid \leq \frac{c}{•2ab}$ and Theorem  \ref{k1} gives $\mid a_{3} \mid \leq \frac{(c)_{2}}{(a)_{2}(b)_{2}}$. The estimate for $\mid a_{2} \mid$ is sharp when $f$ is defined by
    \begin{equation*}\label{p29}
        f(z)= z _{1}F_{2}(c,b,a;z)* \frac{-z}{2} (1+z)^{-1}\qquad(z \in \mathcal{U}),
    \end{equation*}
    and the estimate for $\mid a_{3} \mid $ is sharp for the function $g$ defined by
    \begin{equation*}\label{p30}
        g(z)=z _{1}F_{2}(c,b,a;z)* 4z\sqrt{1+z^{2}}\qquad(z \in \mathcal{U}).
    \end{equation*}
\end{corollary}
We will proceed  this Theorem in the case of $\mu \in \mathbb{R}$.
\begin{theorem}\label{k7}
    Let $\mu \in \mathbb{R}$. If the function $f$ given by $(\ref{p1})$ belongs to the class $\mathcal{R}_{a,b}^{c}$, then
    \begin{eqnarray*}\label{p31}
        \mid a_{3}-\mu a_{2}^{2}\mid \leq \begin{cases}\frac{(c)_{2}}{2(a)_{2}(b)_{2}}\left(\frac{{-(c+1)ab+\mu c
        (a+1)(b+1)}}{2ab(c+1)}\right);
            &\text{}~~ \mu <\frac{-5(c+1)ab}{c(a+1)(b+1)},\\\frac{(c)_{2}}{(a)_{2}(b)_{2}};&\text{}~~\frac{-5(c+1)ab}{c(a+1)(b+1)}\leq \mu\leq \frac{3(c+1)ab}{c(a+1)(b+1)},\\ \frac{(c)_{2}}{2(a)_{2}(b)_{2}} \left(\frac{(c+1)ab+\mu (a+1)(b+1)c}{2ab(c+1)}\right);&\text{}~~\mu>\frac{3ab(c+1)}{c(a+1)(b+1)}.\end{cases}\end{eqnarray*}
    The estimate is sharp for the functions $f$ defined in $\mathcal{U}$ by
    \begin{equation*}\label{p32}
        f(z)=\begin{cases}z _{1}F_{2}(c,b,a;z)* 4z\sqrt{1+z^{2}};
            &\text{}~~\frac{-5(c+1)ab}{c(a+1)(b+1)}\leq \mu\leq \frac{3(c+1)ab}{c(a+1)(b+1)},\\ z _{1}F_{2}(c,b,a;z)* \frac{-z}{2}(1+z)^{-1};&\text{}~~ \mu <\frac{-5(c+1)ab}{c(a+1)(b+1)}~~ \text{or}~~ \mu>\frac{3ab(c+1)}{c(a+1)(b+1)}. \end{cases}
    \end{equation*}
    \begin{proof}
        Since from  $(\ref{p27})$, we have
        \begin{equation*}\label{r1}
            \mid a_{3}- \mu  a_{2}^{2}\mid=\frac{1}{2}\frac{(c)_{2}}{(a)_{2}(b)_{2}}\left|p_{2}-\frac{5ab(c+1)+\mu c (a+1)(b+1)}{8ab(c+1)}p_{1}^{2}\right|.
        \end{equation*}
        The result follows upon applications of  Lemma \ref{b1} in
        (\ref{p27}).
        This completes the proof of Theorem \ref{k7}.
    \end{proof}
\begin{remark}
Putting $b=1$ in the Theorem $\ref{k7}$, we obtained the recent
result, due to Patel and Sahoo (cf. \cite{22a}, Corollary 5).
\end{remark}
\end{theorem}
In the following theorem, we find sharp upper bound to the second Hankel determinant for the class $\mathcal{R}_{a,b}^{c}.$
\begin{theorem}\label{k8}
    If $a\geq c\geq 1/2 $ and the function $f$ given by $(\ref{p1})$ belongs to the class $\mathcal{R}_{a,b}^{c},$ then
    \begin{equation}\label{p33}
        |a_{2}a_{4}-a_{3}^{2}|\leq \left\{\frac{(c)_{2}}{(a)_{2}(b)_{2}}\right\}^{2}.
    \end{equation}
    The estimate in $(\ref{p33})$ is sharp for the functions $g,$ given by$$g(z)=z _{1}F_{2}(c,b,a;z)* 4z\sqrt{1+z^{2}}~~~~~~~~~~~~\qquad(z \in \mathcal{U}).$$
    \begin{proof}
        Let the function $f \in \mathcal{R}_{a,b}^{c} .$ From $(\ref{p24}),$ $(\ref{p25})$ and $(\ref{p26}),$ we get
        \begin{equation}\label{p34}
            |a_{2}a_{4}-a_{3}^{2}|=\frac{3}{8}\left|\frac{c}{a}\frac{(c)_{3}}{(a)_{3}b(b)_{3}}\left(p_{1}p_{3}-\frac{5}{4}p_{1}^{2}p_{2}+\frac{13}{32}p_{1}^{4}\right)-\left\{\frac{(c)_{2}}{(a)_{2}(b)_{2}}\right\}^{2}\frac{1}{4}\left(p_{2}^{2}-\frac{5}{4}p_{1}^{2}p_{2}+\frac{25}{64}p_{1}^{4}\right)\right|.
        \end{equation}
        Since the function $\phi(z)\in \mathcal{P}.$ We assume that $p_{1}>0$ and $p_{1}=p~~(0\leq p\leq2).$ Now by using $(\ref{p13})$ and $(\ref{p14})$ in $(\ref{p34}),$ we get
        \begin{eqnarray}\label{p35}
            |a_{2}a_{4}-a_{3}^{2}|&=& \left| \frac{3}{32}\frac{c}{a}\frac{(c)_{3}}{(a)_{3}b(b)_{3}}{ \left\{p^{4}+2(4-p^{2})p^{2}x-(4-p^{2})p^{2}x^{2}+2(4-p^{2})(1-\mid x \mid ^{2})pz \right\}}
            \right.\nonumber\\&& \left. +\frac{5c(c)_{2}}{32a(a)_{2}b(b)_{2}}\left\{\frac{-3(c+2)(a+1)(b+1)+2(c+1)(a+2)(b+2)}{2(a+1)(a+2)(b+1)(b+2)}\right\}\right.\nonumber\\&& \left.\{p^{4}+(4-p^{2})p^{2}x\}-\frac{1}{16}\left\{\frac{(c)_{2}}{(a)_{2}(b)_{2}}\right\}^{2}\{p^{4}+(4-p^{2})^{2}x^{2}+2p^{2}(4-p^{2})x\} \right.\nonumber\\&&
            \left.+\frac{(39(c+2)(a+1)(b+1)-25(c+1)(a+2)(b+2))c(c)_{2}}{256(a)_{3}(a)_{2}(b)_{3}(b)_{2}}p^{4}\right|,
        \end{eqnarray}
        for some $x~~(\mid x \mid \leq1)$ and for some $z~~(\mid z \mid \leq1)$. Applying the triangle inequality in $(\ref{p35})$ and replacing  $ \mid x \mid$ by $y$ in the equation , we get
        \begin{eqnarray}\label{p36}
            |a_{2}a_{4}-a_{3}^{2}|&\leq & \frac{c(c)_{2}}{16 a (a)_{2}b(b)_{2}}\left\{\frac{2abc+ac+bc-c+5ab+4a+4b+2}{16(a+1)_{2}(b+1)_{2}}\right\}p^{4}+\frac{c(c)_{2}}{64(a)_{2}(a)_{3}(b)_{2}(b)_{3}}\nonumber\\&&(4-p^{2})p^{2}y\{abc-ac-bc-5c+4ab+2a+2b-2\} +\frac{(4-p^{2})c(c)_{2}}{32(a)_{2}(a)_{3}(b)_{2}(b)_{3}}y^{2}\nonumber\\&&\{3(c+2)(a+1)(b+1)3p(p-2)+2(4-p^{2})(c+1)(a+2)(b+2)\}+\nonumber\\&&\frac{3c(c)_{3}}{16(a)_{3}a(b)_{3}b}(4-p^{2})p:= G(p,y)~~~~(0\leq p\leq2 ; 0\leq y\leq 1).~~(say)
        \end{eqnarray}
        Next we maximize the function  $G(p,y)$ on the closed rectangle $[0,2]\times [0,1].$ 
Indeed, for $0 \leq p \leq 1$ and $0 \leq y \leq 1,$ we have
        \begin{eqnarray*}\label{p37}
            \frac{\partial G}{\partial y}&=&\frac{c(c)_{2}(4-p^{2})}{8(a)_{2}(a)_{3}(b)_{2}(b)_{3}}\Big\{\frac{p^{2}}{8}(abc-ac-bc-5c+4ab+2a+2b-2)
            \\ & +&y
            \Big(3(c+2)(a+1)(b+1)3p(p-2)+2(4-p^{2})(c+1)(a+2)(b+2)\Big) \Big\} > 0.
        \end{eqnarray*}
  Which clearly shows that $G(p,y)$ cannot attain maximum in the interior of the closed rectangle $[0,2] \times [0,1].$
        \begin{equation*}\label{p38}
            \max_{0\leq y \leq 1} G(p,y)=G(p,1)=F(p) ~~~~~(say),
        \end{equation*}
        Therefore maximum must be attain on the boundary.Thus for fixed $p,~~~(0 \leq p \leq 2),$ we have 
        \begin{eqnarray}\label{p39}
                    F(p)&=&\frac{c(c)_{2}p^{4}}{16 a (a)_{2}b(b)_{2}}\left\{\frac{-10abc+13ac+13bc+59c-43ab-20a-20b+26}{•16(a+1)(a+2)(b+1)(b+2)}\right\}\nonumber\\ &+&\frac{c(c)_{2}p^{2}}{2 a (a)_{2}b(b)_{2}}\left\{\frac{-3abc-9ac-11bc-31c+6a-6b-18}{•8(a+1)(a+2)(b+1)(b+2)}\right\}+ \frac{c(c)_{2}(c+1)}{ (a)_{2} (a)_{2}(b)_{2}(b)_{2}}.
        \end{eqnarray}
   Differentiating (\ref{p39}) partially w.r.t. $p$ and equating to zero yields,
        \begin{multline*}\label{p40}
            \frac{\partial F}{\partial p}=\frac{c(c)_{2}}{ a (a)_{2}b(b)_{2}}p\Bigg[\left\{\frac{-10abc+13ac+13bc+59c-43ab-20a-20b+26}{64(a+1)_{2}(b+1)_{2}}\right\}p^{2} \\+\left\{\frac{-3ab-9ac-11bc-31c+6a-6b-18}{8(a+1)_{2}(b+1)_{2}}\right\}\Bigg]=0,
        \end{multline*}
       which implies that either $p=0$ or

        $$p^{2}=\frac{8(-3abc-9ac-11bc-31c+6a-6b-18)}{(-10abc+13ac+13bc+59c-43ab-20a-20b+26)}.$$
         Further, we have  $F^{\prime \prime}(0)<0.$ Thus the maximum value of $F$ is attained at $p=0.$ Therefore,  the upper bound in $(\ref{p36})$ corresponds to $p=0$ and $y=1$ becames
        $$|a_{2}a_{4}-a_{3}^{2}|\leq \left\{\frac{(c)_{2}}{(a)_{2}(b)_{2}}\right\}^{2}.$$
        This completes the proof of Theorem \ref{k8}.
    \end{proof}
\end{theorem}
\begin{remark}\label{k9}
    Putting $b=1$ in the Theorem $\ref{k8},$ we obtained the recent result, due to Patel and Sahoo (cf. \cite{22a}, Theorem 7).
\end{remark}
Next we find the sharp upper bound for the fourth co-efficient of functions belonging to the class $\mathcal{R}_{a,b}^{c}.$
\begin{theorem}\label{k10}
    If the function $f$ given by $(\ref{p1})$ belongs to the class $\mathcal{R}_{a,b}^{c}$ then
    \begin{equation}\label{p41}
        \mid a_{4} \mid\leq \frac{3(c)_{3}}{•(a)_{3}(b)_{3}}.
    \end{equation}
    The estimate (\ref{p41}) is sharp.
    \begin{proof}
        Using $(\ref{p14})$ in $(\ref{p26}),$ we assume that $p_{1}>0$ and write $p_{1}= p~~ (0\leq p\leq2),$ then we deduce that
        \begin{multline}\label{p42}
            \mid a_{4}\mid=\left|\frac{6(c)_{3}}{•(a)_{3}(b)_{3}}\left(\frac{p^{3}}{128•}-\frac{(4-p^{2})px}{•32}-\frac{1}{•16}(4-p^{2})px^{2}+\frac{1}{•8}(4-p^{2})(1-\mid x \mid^{2})z\right) \right|,
        \end{multline}
        for some $x~~~~ (\mid x \mid \leq1)$ and for some $z~~~~ (\mid z \mid \leq1).$ Applying the triangle inequality  and replace $\mid x \mid$ by $y$ in (\ref{p42}), we get
        \begin{multline}\label{p43}
            \mid a_{4} \mid \leq \frac{6(c)_{3}}{•(a)_{3}(b)_{3}}\left(\frac{p^{3}}{128•}+\frac{(4-p^{2})py}{•32}+\frac{1}{•16}(4-p^{2})py^{2}+\frac{1}{•8}(4-p^{2})(1-\mid y \mid^{2})z\right)\\= G(p,y),~~~~~~~~~~~~~(0 \leq y \leq 1; 0 \leq  p \leq 2)~~~~(say).
        \end{multline}
        We next maximize the function $ G(p,y)$ on the closed rectangle $[0,2]\times [0,1].$ Since
        \begin{eqnarray*}\label{p44}
            G^{\prime}(y)=\frac{6(c)_{3}}{•(a)_{3}(b)_{3}}\left(\frac{(4-p^{2})p}{•32}+\frac{1}{8}(4-p^{2})py-\frac{1}{4}(4-p^{2})y\right)<0,
        \end{eqnarray*}
        for $0<p<2$ and $0<y <1$; it follows that $G(p,y)$ can't have a maximum  value in the interior of the closed rectangle $[0,2]\times [0,1].$ Thus, for fixed $p \in [0,2],$
        \begin{eqnarray*}\label{p45}
            \max_{0\leq y\leq1} G(p,y)=G(p,0)=F(p),
        \end{eqnarray*}
        where
        \begin{eqnarray*}\label{p46}
            F(p)=\frac{6(c)_{3}}{•(a)_{3}(b)_{3}}\left(\frac{p^{3}}{•128}+\frac{1}{•2}-\frac{p^{2}}{•8}\right).
        \end{eqnarray*}
        We further note that,
        \begin{eqnarray}\label{p47}
            F^{\prime}(p)=\frac{6(c)_{3}}{•(a)_{3}(b)_{3}}\left(\frac{3p}{•128}-\frac{1}{•4}\right)p,
        \end{eqnarray}
        for $p=0$ or $p=32/•3.$ Since
        $$F^{\prime\prime}(0)=\frac{-3(c)_{3}}{•2(a)_{3}(b)_{3}}<0,$$
        the function $F$ attains maximum value at $p=0.$ Thus, the upper bound of the function $G$  corresponding  to $p=y=0.$ Therefore,  putting $p=y=0$ in $(\ref{p43}),$ we get
        \begin{equation}\label{p48}
            \mid a_{4} \mid\leq \frac{3(c)_{3}}{•(a)_{3}(b)_{3}}.
        \end{equation}
        The estimate in $(\ref{p48})$ is sharp for the function $f$ defined by,
        $$f(z)=z _{1}F_{2}(c,b,a;z)* 36z\sqrt{ 1+z^{3}}.$$
        This completes the proof of Theorem \ref{k10}.
    \end{proof}
    \begin{remark}\label{neha}

    Putting $b=1$ in the Theorem \ref{k10}, we obtained the recent result due to Patel and Sahoo (cf. \cite{22a}, Theorem 9).
    \end{remark}
\end{theorem}
\begin{theorem}\label{a3}
If the function $f$ given by $(\ref{p1})$ belongs to the class $\mathcal{R}_{a,b}^{c},$ then
 \begin{equation*}
        |a_{5}|\leq \frac{15}{16}\frac{(c)_{4}}{(a)_{4} (b)_{4}} .
    \end{equation*}
\end{theorem}
Putting $b=1$ in the Theorem \ref{a3}, we get following:
\begin{corollary}
If the function $f$ given by $(\ref{p1})$ belongs to the class $\mathcal{R}{(a,b)},$ then
\begin{equation*}
        |a_{5}|\leq \frac{15}{384}\frac{(c)_{4}}{(a)_{4}} .
    \end{equation*}
\end{corollary}

The following Theorems are straight forward verification on applying the same procedure as described in Theorem \ref{k8}.
\begin{theorem}\label{a1}
If the function $f$ given by $(\ref{p1})$ belongs to the class $\mathcal{R}_{a,b}^{c},$ then
 \begin{equation*}
        |a_{2}a_{3}-a_{4}|\leq \frac{13(c)_{2}}{64 a (a)_{3} b (b)_{3}}
        \left(\frac{c(a+2)(b+2)+9ab(c+2)}{c(a+2)(b+2)+11ab(c+2)}\right)^{\frac{1}{2}}(c(a+2)(b+2)+9ab(c+2)).
    \end{equation*}
\end{theorem}

Putting $b=1$ in the Theorem \ref{a1}, we get following result for the function class $\mathcal{R}(a,c).$
\begin{corollary}
	If the function $f$ given by $(\ref{p1})$ belongs to the class $\mathcal{R}(a,c),$ then
\begin{equation*}
        |a_{2}a_{3}-a_{4}|\leq \frac{13(c)_{2}}{384 a (a)_{3}}
        \left(\frac{3c(a+2)+9a(c+2)}{3c(a+2)+11a(c+2)}\right)^{\frac{1}{2}}(3c(a+2)+9a(c+2)).
    \end{equation*}
\end{corollary}

Next we find the sharp upper bound for third hankel determinant of functions belonging to the class $\mathcal{R}_{a,b}^{c}.$
\begin{theorem}\label{a4}
If the function $f$ given by $(\ref{p1})$ belongs to the class $\mathcal{R}_{a,b}^{c},$ then
\begin{multline*}
        |H_{3}(1)|\leq \left(\frac{(c)_{2}}{(a)_{2}(b)_{2}}\right)^{3}+\frac{39(c)_{2}(c)_{3}}{64((a)_{3})^{2}((b)_{3})^{2}} \left(\frac{c(a+2)(b+2)+9ab(c+2)}{c(a+2)(b+2)+11ab(c+2)}\right)^{\frac{1}{2}}\\ \times(c(a+2)(b+2)+9ab(c+2))+\frac{15(c)_{4}(c)_{2}}{64(a)_{2}(a)_{4}(b)_{2}(b)_{4}}.
    \end{multline*}
\begin{proof}

Suitable applications of 
 Theorems \ref{k8}, \ref{k10},  \ref{a1}, \ref{a2},  \ref{a3} and Corollary \ref{k6} in equation (\ref{anuja}), the result follows.
    This completes the proof of Theorem \ref{a4}.
\end{proof}
\end{theorem}
Putting $b=1$ in the Theorem \ref{a4}, we obtained the following third Hankel determinant for the function class $\mathcal{R}(a,c).$
\begin{corollary}
If the function $f$ given by (\ref{p1}) belongs to the class $\mathcal{R}(a,c)$ then
\begin{multline*}
        |H_{3}(1)| \leq \left(\frac{(c)_{2}}{2(a)_{2}}\right)^{3}+\frac{39(c)_{2}(c)_{3}}{2304((a)_{3})^{2}} \left(\frac{3c(a+2)+9a(c+2)}{3c(a+2)+11a(c+2)}\right)^{\frac{1}{2}} \\ \times(3c(a+2)+9a(c+2))+\frac{15(c)_{4}(c)_{2}}{3072(a)_{2}(a)_{4}}.
\end{multline*}

\end{corollary}
Finally, we have following sufficient condition for a function in $\mathcal{A}$ to be in the class $\mathcal{R}_{a,b}^{c}:$
\begin{theorem}\label{k11}
    Let $\gamma >0.$ If $f \in \mathcal{A}$ satisfies
    \begin{equation*}\label{p49}
        \Re \left\{\frac{\mathcal{I}_{a+1,b}^{c}f(z)}{\mathcal{I}_{a,b}^{c}f(z)•}\right\}<1+\frac{1}{•2a \gamma}~~~~~~~~\qquad(z \in \mathcal{U}),
    \end{equation*}

    then
    \begin{eqnarray*}\label{p50}
        \frac{\mathcal{I}_{a,b}^{c}f(z)}{•z}\prec (1+z)^{1/\gamma}~~~~~~~~\qquad(z \in \mathcal{U})
    \end{eqnarray*}
    and the result is the best possible.
    \end{theorem}
    \begin{proof}
        Setting
        \begin{eqnarray}\label{p51}
            \frac{\mathcal{I}_{a,b}^{c}f(z)}{•z}= (1+w(z))^{1/\gamma}~~~~~~~~\qquad(z \in \mathcal{U}).
        \end{eqnarray}
        Choosing the principal branch in $(\ref{p51}),$ we see that $w$ is analytic in $\mathcal{U}$ with $w(0)=0.$ Taking the logarithmic differentiation in $(\ref{p51})$ and using the identity $(\ref{p5})$ in the resulting equation, we deduce that
        \begin{eqnarray}\label{p52}
            \frac{(\mathcal{I}_{a+1,b}^{c}f(z))}{•(\mathcal{I}_{a,b}^{c}f(z))}=1+\frac{zw^{\prime}(z)}{•a \gamma (1+w(z))}~~~~~\qquad(z \in \mathcal{U}).
        \end{eqnarray}
        Next to claim that $ \mid w(z) \mid <1,$ $z \in \mathcal{U}.$ $\exists$ a $z_{0} \in \mathcal{U}$ such that
        \begin{eqnarray*}\label{p53}
            \max_{\mid z \mid \leq \mid z_{0} \mid} \mid w(z) \mid= \mid w(z_{0}) \mid = 1~~~~~(w(z_{0})\neq 1).
        \end{eqnarray*}
        Letting $w(z_{0})=e^{i\theta}~~(-\pi <\theta \leq \pi )$ and applying Jack's lemma \cite{8}, we have
        \begin{eqnarray}\label{p54}
            z_{0}w^{\prime}(z_{0})=k w(z_{0})~~~~~~~\qquad (k\geq 1).
        \end{eqnarray}
        Using $(\ref{p54})$ in $(\ref{p52})$ then, we get
        \begin{eqnarray}\label{p55}
            \Re \left( \frac{\mathcal{I}_{a+1,b}^{c}f(z)}{•\mathcal{I}_{a,b}^{c}f(z)} \right)&=&1+\frac{1}{•a \gamma} \Re \left\{\frac{z_{0}w^{\prime}(z_{0})}{•1+w(z_{0})}\right\}\nonumber \\&=& 1+\frac{k}{•a\gamma}\Re \left(\frac{•e^{i\theta}}{•1+e^{i\theta}}\right)\nonumber \\ &\geq & 1+\frac{k}{•2a \gamma}.
        \end{eqnarray}
        Thus we conclude that $\mid w(z) \mid <1$ for $z \in \mathcal{U}$ and  the theorem follows from $(\ref{p51}).$ For sharpness we consider  the principal branch of the function $f_{0}$ defined as,
        \begin{eqnarray}\label{p56}
            f_{0}(z)=z~~ _{1}F_{2}(c,b,a;z)\ast z (1+z)^{\frac{1}{•\gamma}}~~~~\qquad(z \in \mathcal{U}).
        \end{eqnarray}
       Therefore   $(\ref{p56}),$  yields
        \begin{eqnarray*}\label{p57}
            \frac{\mathcal{I}_{a,b}^{c}f_{0}(z)}{•z}= (1+z)^{1/\gamma}.
        \end{eqnarray*}
       Taking logarithmic differentiation and suitable application of \ref{p5}, gives
        \begin{eqnarray*}\label{p58}
            \Re \left( \frac{\mathcal{I}_{a+1,b}^{c}f_{0}(z)}{•\mathcal{I}_{a,b}^{c}f_{0}(z)} \right)&=& 1+\frac{1}{•a\gamma}\frac{z}{•1+z}\nonumber \\ &\longrightarrow & 1+\frac{1}{•2a \gamma}~~as ~~ z\longrightarrow 1^{-}.
        \end{eqnarray*}
        This completes the proof of Theorem \ref{k11}.
    \end{proof}
    \begin{remark}
Putting $b=1$ in the theorem $\ref{k11},$ we get the result due to Patel and Sahoo (cf. \cite{22a}, Theorem 11).

    \end{remark}
    Putting $\gamma=2$ in Theorem \ref{k11}, we have the following
    \begin{corollary}
        If $f \in \mathcal{A}$ satisfies
        \begin{equation*}\label{p59}
            \Re \left\{\frac{\mathcal{I}_{a+1,b}^{c}f(z)}{\mathcal{I}_{a,b}^{c}f(z)}\right\}<1+\frac{1}{4a}\qquad(z \in \mathcal{U}),
        \end{equation*}
        then $f \in \mathcal{R}_{a,b}^{c}$. The result is the best possible.
    \end{corollary}


\begin{thebibliography}{99}
\bibitem{1aa} A. Abubaker and M. Darus,  \newblock Hankel determinant for a class of analytic
functions involving a generalized linear differential operator,
\newblock{\em Int. J. Pure Appl. Math.}, {\bf 69}, (2011), 429–-435.
\bibitem{2b}  K. O. Babalola,  \newblock On $H_{3}(1)$ Hankel determinant for some classes of univalent functions,
\newblock{\em Inequal. Theory Appl.} {\bf 6}, (2007), 1--7.
\bibitem{2bb} D. Bansal, S. Maharana and  J. K. Prajapat,  \newblock Third order Hankel determinant for certain univalent functions,
\newblock{\em J. Korean Math. Soc.}, {\bf 52 (6)}, (2015), 1139–-1148.

    \bibitem{1} S. D. Bernardi, \newblock  Convex and Starlike univalent functions, \newblock{\em  Trans. Amer. Math. Soc.}, {\bf  135}, (1969),  429--446.
\bibitem{1a} B. Bhowmik, S. Ponnusamy and K. -J. Wirths, \newblock On the Fekete- Szeg\"{o}
problem for concave univalent functions, \newblock{\em J. Math.
Anal. Appl.},  {\bf 373(2)}, (2011), 432– 438.
    \bibitem{2} B. C. Carlson and D. B. Shaffer, \newblock Starlike and prestarlike hypergeometric functions, \newblock{\em SIAM Journal Anal.}, {\bf 15}, (1984), 737--745.
\bibitem{cho} N. E. Cho, B. Kowalczyk and A. Lecko,  \newblock Fekete-Szeg\"{o} problem for close-to-convex funtions with respect to a certain convex function depend on a real parameter, \newblock{\em  Front. Math. China}, {\bf  11(6)} (2016), 1471--1500.

    \bibitem{3} J. H. Choi, M. Saigo and H. M. Srivastava, \newblock Some inclusion properties of a certain family of integral operators, \newblock{\em J. Math. Anal. Appl.}, (2002), 432--445.
\bibitem{dziok} J. Dziok,  \newblock A general solution of Fekete-Szeg\"{o} problem, \newblock{\em  Bound. Value Probl.}, {\bf  1(98)}, (2013), 1--13.
    \bibitem{4} P. L. Duren, \newblock  Univalent functions,  \newblock{\em Grundlehrender Mathematischer Wissencchaffer}, {\bf 259}, Springer, New york, (1983).
    \bibitem{5} M. Fekete and G. Szego,  \newblock Eine bemerkung uber ungerade schlichte funktionen, \newblock{\em  J. Lond. Math.  Soc}, {\bf 8}, (1933), 85--89.
    \bibitem{6} Y. E. Hohlov, \newblock Hadamard convolution, hypergeometric functions and linear operators in the class of univalent functions, \newblock{\em Dokl. Akad. Nauk Ukr. SSR, Ser.}, {\bf A(7)}, (1984), 25--27.
    \bibitem{7} Y. E. Hohlov, \newblock Convolution operators preserving univalent functions, \newblock{\em Ukrainian Math. J. }, {\bf 37}, (1985), 220--226.
    \bibitem{8} I. S. Jack, \newblock Functions starlike and convex of order $\alpha$, \newblock{\em J.  Lond.  Math.  Soc.}, {\bf 3}, (1971), 469--474.
    
    
    
\bibitem {77}A. Janteng, S. A. Halim and M. Darus, \emph{Coefficient inequality for a function whose derivative has positive real part},
\textsl{J. Inequal. Pure Appl. Math.,} {\bf 7(2)}, 50 (2006).
\bibitem {777}A. Janteng, S. A. Halim and M. Darus, \emph{Hankel deteminant for starlike and convex functions},
\textsl{Int. J. Math. Anal.,} {\bf 1(13)}, (2007), 619--625.

\bibitem{kanas} S. Kanas and  H. E. Darwish,  \newblock Fekete-Szeg\"{o} problem for starlike and convex functions of
complex order, \newblock{\em  Appl. Math. Letters}, {\bf (23)},
(2010), 777--782.
\bibitem{9} F. R. Keogh and E. P. Merkes, \newblock A co-efficient inequality for certain classes of analytic functions, \newblock{\em Proc.  Amer. Math. Soc.}, {\bf 20}, (1969), 8--12.
\bibitem{10} W. Koepf, \newblock On the Fekete- Szeg\"{o} problem for close to convex functions, \newblock{\em Proc. Amer.   Math. Soc.}, {\bf 101(1)}, (1987), 89--95.
\bibitem{11} W. Koepf, \newblock On the Fekete- Szeg\"{o} problem for close to convex functions II, \newblock{\em Archiv der Mathematik}, {\bf 49(5)}, (1987), 420--433.
\bibitem{11a} B. Kowalczyk and A. Lecko,  \newblock Fekete-Szeg\"{o} inequality  for close-to-convex funtions with respect to a certain starlike function depend on a real parameter, \newblock{\em  J. Inequal. Appl.}, {\bf  1(65)}, (2014), 1--16.

\bibitem{12} D. V. Krishna and T. Ram Reddy,  \newblock Coefficient inequality for certain subclasses of analytic functions associated with Hankel determinant, \newblock{\em Indian J. Pure   Appl. Math.}, {\bf 46(1)}, (2015), 91--106.

\bibitem {lee} S. K. Lee, V. Ravichandran and S. Supramaniam, \newblock Bounds for the second
Hankel determinant of certain univalent functions, \newblock{\em J.
Inequal. Appl.}, {\bf 2013}, (2013), 281.

\bibitem{13} R. J. Libera and E. J. Zlotkiewicz, \newblock Early coefficients of the inverse of a regular convex functions, \newblock{\em Proc.   Amer.  Math.  Soc.}, {\bf 85(2)}, (1982), 225--230.
\bibitem{14}  R. J. Libera and E. J. Zlotkiewicz, \newblock Coefficient bounds for the inverse of a function with derivative in $\mathcal{P}$, \newblock{\em Proc.  Amer.  Math.  Soc.}, {\bf 87(2)}, (1983), 251--257.
\bibitem{15}  W. C. Ma and D. Minda, \newblock A unified treatment of some special classes of univalent functions, \newblock{\em In Proc.  Conference on Complex Analysis (Tianjin, 1992)}, Z. Li, F. Ren, L. Yang and S. Zhang, Eds.,   157--169, \newblock
{\em  International Press, Cambridge, Mass, USA,} (1994).
\bibitem{mac} T. H. MacGregor, \newblock Functions whose derivative have a positive real part,
\newblock{\em Trans. Amer. Math. Soc.}, {\bf 104(3)}, (1962), 532--537.
\bibitem{milmo1}  S. S. Miller  and  P. T. Mocanu, Differential
    Subordinations: Theory  and  Applications,  Series  on Monographs
    \ and \ Textbooks  in  Pure  and  Applied
    Mathematics, {\bf 225}, Marcel Dekker, New York, (2000).
\bibitem{16} A. K. Mishra  and  P. Gochhayat, \newblock Second Hankel determinant for a class of analytic  functions defined by fractional derivative, \newblock{\em Int. J.  Math. Math. Sci.}, (2008).
\bibitem{17} A. K. Mishra and P. Gochhayat, \newblock Applications of the Owa-Srivastava operator to the class of k-uniformly convex functions, \newblock{\em  Fract. Calc.  Appl.  Anal.}, {\bf 9(4)}, (2006), 323--231.
\bibitem{18} A. K. Mishra and P. Gochhayat, \newblock The Fekete-Szeg\"{o} problem for k-uniformly convex functions and for a class defined by the Owa-Srivastava operator, \newblock{\em J.   Math. Anal.  Appl.}, {\bf  397(9)}, (2008),  563--572.
\bibitem{19} A. K. Mishra and P. Gochhayat, \newblock The Fekete-Szeg\"{o} problem for a class defined by an integral operator, \newblock{\em Kodai  Math. J.}, {\bf 33}, (2010),  310--328.
\bibitem{19a} A. K. Mishra and P. Gochhayat, \newblock A coefficient inequality for a subclass
of the Carath\'{e}odory functions defined by conical domains,
\newblock{\em Comput. Math. Appl.}, {\bf 61(9)}, (2011),  2816-- 2820.
\bibitem{kund} A. K. Mishra and S. N. Kund, \newblock The second Hankel determinant for a
class of analytic functions associated with the Carlson-Shaffer
operator, \newblock{\em Tamkang J. Math.}, {\bf44(1)}, (2013),
73–-82.
\bibitem{19b} G. Murugusundaramoorthy and K. Vijaya, \newblock Second Hankel determinant for bi-univalent
analytic functions associated with Hohlov operator, \newblock{\em
Int. J. Anal. Appl.}, {\bf 8(1)}, (2015), 22--29
\bibitem{20} J. W. Noonan and D. K. Thomas,  \newblock On the second Hankel determinant of areally mean p-valent functions, \newblock{\em  Trans.  Amer. Math. Soc.}, {\bf 223} (1976), 337--346.
\bibitem{88} K. I. Noor, \newblock Hankel determinant problem for the class of functions with bounded boundary rotation, \newblock{\em Rev. Roumaine Math. Pures Appl.}, {\bf 28(8)}, (1983), 731--739.
\bibitem{21} K. I. Noor and M. A. Noor,  \newblock  On integral operators, \newblock{\em J.  Anal. Appl.}, {\bf 238}, (1999), 341--352.
\bibitem{22}  S. Owa and H. M. Srivastava, \newblock Univalent and starlike generalized hypergeometric functions, \newblock{\em Canad.  J.  Math.}, {\bf 39(5)}, (1987), 1057--1077.
\bibitem{22a} J. Patel and A. K Sahoo, \newblock On certain subclasses of analytic functions involving Carlson-Shaffer operator and related to Lemniscate of bernoulli,  \newblock{\em J. Complex Anal.}, {\bf 2014}, (2014),  1--7.
\bibitem {24a}  Z. Peng, \newblock On the Fekete-Szeg\"{o} problem for a class of analytic functions, \newblock{\em ISRN Math. Anal.}, {\bf 2014}, (2014), 1--4.
\bibitem{r1} M. Raza, S. N. Malik,  \newblock Upper bound of the third Hankel determinant for a class of analytic functions related with Lemniscate of Bernoulli,
\newblock{\em Int. J. Inequal. Appl.}, {\bf 2013}, (2013), Article 412.
\bibitem {ram} T. R. Reddy and D. Vamshee Krishna, \newblock  Hankel determinant for starlike
and convex functions with respect to symmetric points,
\newblock{\em J. Indian Math. Soc.}, {\bf79}, (2012),  161–-171.

\bibitem{23} St. Ruscheweyh, New criteria for univalent functions,
 {\em Proc. Amer. Math. Soc.}, {\bf 49},  (1975), \ 109--115.
    \bibitem{s1} A. K. Sahoo and J. Patel, \newblock  Hankel determinant for a class of analytic functions related with lemniscate of Bernoulli, \newblock{\em Int. J. Anal. Appl.}, {\bf 6 (2)}, (2014), 170--177.
    \bibitem {24}  H. M. Srivastava, A. K. Mishra  and  M. K. Das \newblock The Fekete-Szeg\"{o} problem for a subclasses of close to convex functions, \newblock{\em Complex Var. Elliptic Equ.}, {\bf 44(2)}, (2001), 145–-163.
\bibitem{owasri} H. M. Srivastava and S. Owa (Editors), \newblock Current Topics
in Analytic Function Theory, \newblock{\em World Scientific
Publishing Company, Singapore, New Jersey, London and Hong Kong}
(1992).
\bibitem {YAVUZ} T. Yavuz, \newblock Second Hankel determinant problem for a certain subclass
of univalent functions, \newblock{\em  Int. J. Math. Anal.}, {\bf
9(10)}, (2015), 493 -- 498.

\bibitem {YAVUZ1} T. Yavuz, Second Hankel determinant for analytic functions defined
by Ruscheweyh derivative, \newblock{\em Int. J. Anal. Appl.}, {\bf
8(1)}, (2015), 63--68.




\end{thebibliography}
\end{document}